\documentclass[a4paper, reqno]{amsart}
\usepackage{enumerate}
\usepackage{color}

\newtheorem{theorem}{Theorem}
\newtheorem{lemma}[theorem]{Lemma}

\theoremstyle{remark}

\newtheorem{remark}[theorem]{Remark}

\begin{document}

\title{Krein's trace theorem revisited}

\author[Potapov]{D.~Potapov}
\address{School of Mathematics and Statistics, University of New South Wales, Kensington, NSW, Australia.}
\email{d.potapov@unsw.edu.au}
\author[Sukochev]{F.~Sukochev}
\email{f.sukochev@math.unsw.edu.au}
\author[Zanin]{D.~Zanin}
\email{d.zanin@math.unsw.edu.au}


\begin{abstract} We supply the first proof of Krein's Trace Theorem which does not use complex analysis.
Our proof holds for~$\sigma$-finite von Neumann algebras
$\mathcal{M}$ of type II and unbounded perturbations from the
predual of~$\mathcal{M}$.
\end{abstract}

\maketitle

\section{Introduction}
\label{sec:intro}

The first attempt to deliver a proof of the existence of the Krein
Spectral Shift Function (KrSSF) without using complex analytical
facts was made by M.Sh.~Birman and M.Z.~Solomyak in 1972
(see~\cite{BS1972-onSSF} and also~\cite{BirPus1998-MR1607900}).
Their method was based on the theory of double operator integrals
developed by those authors
in~\cite{BS-DOI-I,BS-DOI-II,BS-DOI-III}. That attempt led to
introducing of an important notion of the spectral averaging
measure (see also~\cite{Simon}), but was not successful since the
authors of~\cite{BS1972-onSSF}  failed to prove directly the
absolute continuity of that measure with respect to Lebesgue
measure.  The second attempt to deliver such a proof is due to D.
Voiculescu~\cite{Voiculescu}, whose method is based on the usage
of classical Weyl-von Neumann theorem.  However, his attempt also
fails to recover the full generality of Krein's original result.
In our present paper, we combine methods drawn from the double
operator integration theory with Voiculescu's ideas and deliver a
rather short and straightforward new proof of Krein's result in
full generality. The advantage of our present approach is seen
from the following extension.

We deliver the complete proof of the existence of the spectral
shift function~$\xi_{A,B}$ in the setting when~$A$ and~$B$ are
self-adjoint operators affiliated with a~$\sigma$-finite
semifinite von Neumann algebra~$\mathcal{M}$, whose difference
$A-B$ belongs to the predual of~$\mathcal{M}$. If~$\mathcal{M}$ is
a type I factor, then this is precisely the Krein's result. For a
general semifinite von Neumann algebra, the earlier attempt based
on emulating the Krein's complex-analytical proof only yielded the
 result for the special case when~$V$ was necessarily a bounded trace class
perturbation~\cite{ADS}.



\section{Preliminaries}\label{sec:prelim}

Let~$H$ be an infinite dimensional Hilbert space and
let~$\mathcal{L}(H)$ be the algebra of all bounded operators
in~$H$. In what follows, $\mathcal{M}$ is a von Neumann algebra
on~$H$, that is a $*$-subalgebra of~$\mathcal{L}(H)$ closed in the
weak operator topology. The identity in $\mathcal{M}$ is denoted
by $1$. We are only interested in semifinite von Neumann algebras,
that is, those which admit a faithful normal semifinite
trace~$\tau$. We fix a couple~$(\mathcal{M},\tau)$. A von Neumann
algebra is said to be~$\sigma$-finite if it admits at most
countably many orthogonal projections.

An (unbounded) operator is said to be affiliated with
$\mathcal{M}$ if it commutes with every operator in the commutant
$\mathcal{M}'$ of~$\mathcal{M}$. Closed densely defined operator
$A$ affiliated with~$\mathcal{M}$ is said to be~$\tau$-measurable
if, for every~$\varepsilon>0$, there exists a projection
$p\in\mathcal{M}$ such that~$\tau(p)<\varepsilon$ and such that
$(1-p)H\subset {\rm dom}(A)$. The collection of all
$\tau$-measurable operators is denoted by $S(\tau)$. The real
vector space $S_{h}\left( \tau \right) =\left\{ A\in S\left( \tau
\right) :A=A^{\ast }\right\} $ is a partially ordered vector space
with the ordering defined by setting $A\geq 0$ if and only if
$\left\langle A\xi ,\xi \right\rangle \geq 0$ for all $\xi \in
\mathcal{D}\left( A\right) $. The positive cone in $S_{h}\left(
\tau \right) $ will be denoted by $S\left( \tau \right) ^{+}$. The
positive part $A_{+}$ and negative part $A_{-}$ of an operator
$A\in S_{h}\left( \tau \right) $ are defined by
\begin{equation*}
A_{+}:=\int_{\mathbb{R}}\lambda ^{+}dE_A(-\infty,\lambda] \text{ \
\ \ and \ \ \ }A_{-}:=\int_{\mathbb{R}}\lambda
^{-}dE_A(-\infty,\lambda]
\end{equation*}
respectively, where $\lambda ^{+}=\max \left( \lambda ,0\right) $
and $ \lambda ^{-}=\max \left( -\lambda ,0\right) $ and
$E_A(-\infty,\lambda]$ is the spectral projection of the
self-adjoint operator $A$ corresponding to the
interval~$(-\infty,\lambda]$. It follows immediately from the
spectral theorem that $ A=A_{+}-A_{-}$.

The notions of the distribution function~$n_A$, $A\in S_{h}\left(
\tau \right) $ and that of the singular value function~$\mu(A)$,
$A\in S\left( \tau \right) $ are defined as follows
$$n_A(t):=\tau(E_A(t,\infty)),\ t\in\mathbb{R}  \ \ \text{and}\ \
\mu(t;A):=\inf\{s: n_{|A|}(s)\leq t\},\quad t\ge0.$$ It follows
directly that the singular value function $\mu(A)$ is a
decreasing, right-continuous function on the positive half-line
$[0,\infty)$. The trace $\tau $ extends to $S\left( \tau \right)
^{+}$ as a non-negative extended real-valued functional which is
positively homogeneous, additive, unitarily invariant and normal.
This extension is given by
\[
\tau \left( A\right) =\int_{0}^{\infty }\mu \left( t;A\right) dt,\
\ \ A\in S\left( \tau \right) ^{+},
\]
and satisfies $\tau \left( A^{\ast }A\right) =\tau \left( AA^{\ast
}\right) $ for all $A\in S\left( \tau \right) $. If ${\mathcal M}
={\mathcal L}({\mathcal H})$ and $\tau $ is the standard trace,
then $S(\tau) ={\mathcal M}$. In this case, an operator $A\in
S(\tau)$ is compact if and only if $\lim_{t\to \infty}\mu(t;A)=0$
and
$$
\mu _n(A)=\mu \left(t;A\right), \quad t\in [n,n+1),\quad
n=0,1,2,\dots ,
$$
and the sequence $\{\mu _n(A)\}_{_{n=0}}^{\infty }$ is just the
sequence of eigenvalues of $\vert A\vert $ in non-increasing order
 and counted according to multiplicity.

The noncommutative
space~$\mathcal{L}_p=\mathcal{L}_p(\mathcal{M},\tau)$, $1\leq
p\leq\infty$ is defined as follows
$$\mathcal{L}_p=\{A\in S(\tau):\ \mu(|A|)\in L_p:=L_p(0,\infty)\},$$
where $L_p(0,\infty)$ is the usual Lebesgue space. The space
$\mathcal{L}_p $ is a linear subspace of $S\left( \tau \right) $
and the functional $A\longmapsto \left\Vert A\right\Vert
_{p}:=\left (\tau(|A|^p)\right)^{1/p}$, $A\in \mathcal{L}_p $,
$1\leq p<\infty$ is a norm. For convenience, we set
$(\mathcal{L}_\infty,\|\cdot\|_\infty)=\mathcal M$ equipped with
the uniform operator norm. We have, in particular, $\left\Vert
A\right\Vert _{p}=\left\Vert \mu \left( A\right) \right\Vert _{p}$
for all $A\in \mathcal{L}_p $, $1\leq p\leq \infty$ (we denote the
norm on $\mathcal{L}_p$ and the standard norms on Lebesgue spaces
$L_p(0,\infty)$ and $L_p(\mathbb{R})$ by the same symbol
$\|\cdot\|_p$ and this should not cause any confusion). We recall
the following useful formula
$$\|n_A\|_1=\int_0^{\infty}n_A(s)ds = sn_A(s)|_0^{\infty}-\int_0^{\infty}sdn_A(s)=\tau(A)=\|A\|_1,$$
which holds for every ~$A\in S\left( \tau \right) ^{+}$. Equipped
with the norm $\|\cdot\|_1$, the space $\mathcal{L}_1$ is a Banach
space. It is well known (see e.g.~\cite{PX}) that $\mathcal{L}_1$
is isometric to a predual~$\mathcal{M}_{*}$ of the von Neumann
algebra~$\mathcal{M}$. In what follows, we will need the following
result whose proof follows verbatim from that of ~\cite[Lemmas 15
and 16]{SZ}, where this result stated for $s,t>0$.

\begin{lemma}\label{lemma} Suppose that $\mathcal{M}$ is a finite von Neumann algebra and that
~$\tau(1)<\infty$. If operators $A, B \in S_{h}\left( \tau \right)
$ satisfy $A\geq B$, then $n_A\geq n_B$ and the inequality
\begin{equation}\label{nsum}
n_{A+B}(s+t)\leq n_A(s)+n_B(t)
\end{equation}
holds for all $s,t\in \mathbb{R}$.
\end{lemma}

The following Weyl-von Neumann type theorem is at the core of our
approach in this paper. It is (implicitly) proved in~\cite{PSW}.
For the classical Weyl-von Neumann theorem we refer the reader to~\cite{Davidson}.

\begin{theorem}\label{wvn} Let~$\mathcal{M}$ be a~$\sigma$-finite von Neumann algebra equipped
with a faithful normal semifinite trace~$\tau$. For every
$A\in\mathcal{M}$, there exists a sequence of~$\tau$-finite
projections~$p_n\uparrow 1$, $n\geq0$, such that
$\|[A,p_n]\|_2\to0$.
\end{theorem}
\begin{proof} It follows from Lemma 6.4 in~\cite{PSW} that there exists a net~$p_i$, $i\in\mathbb{I}$,
 of orthogonal projections such that the algebra~$p_i\mathcal{M}p_i$ admits a~$\tau$-finite generating projection.
 Since~$\mathcal{M}$ is~$\sigma$-finite, it follows that the set~$\mathbb{I}$ is, at most, countable.
 The assertion follows now from Lemma 6.3 in~\cite{PSW}.
\end{proof}

The following two lemmas constitute a small complement to
Theorem~\ref{wvn}. These lemmas are at the core of Voiculescu's
approach~\cite{Voiculescu}. If $A\in S_h(\tau)$, then the
projection onto the closure of the range of $\vert A\vert$ is
called the support of $A$ and is denoted by ${\rm supp}(A)$.

\begin{lemma}\label{pr lemma} Let~$\mathcal{M}$ be a von Neumann algebra equipped with a faithful normal
semifinite trace~$\tau$. If  $p_n\uparrow 1$ and if
$C\in\mathcal{M}$, $C^*=C$ is such that~$\tau({\rm
supp}(C))<\infty$, then~$\|[C,p_n]\|_2\to 0$.
\end{lemma}
\begin{proof} Define~$A_n={\rm supp}(C)p_n{\rm supp}(C)$. We have~$A_n\uparrow {\rm supp}(C)$ and
$$\|[C,p_n]\|_2^2=-\tau([C,p_n]^2)=2\tau(A_nC^2)-2\tau((A_nC)^2).$$
The sequence~$A_n$ strongly converges to~${\rm supp}(C)$.
Therefore, $A_nC^2\to C^2$ and~$A_nC\to C$ strongly.  By
~\cite[Proposition 2.4.1]{BratRob}, we have~$(A_nC)^2\to C^2$
strongly. Since the trace $\tau$ is strongly continuous (see e.g.
\cite[Lemma~1.2, Theorem~1.10]{SZ-Lec-on-vNA}) and the algebra
~${\rm supp}(C)\mathcal{M}{\rm supp}(C)$ is finite, it follows
that $\tau(A_nC^2)\to\tau(C^2)$ and~$\tau((A_nC)^2)\to\tau(C^2)$.
This suffices to conclude the proof.
\end{proof}

\begin{lemma}\label{voic lemma} Let~$\mathcal{M}$ be a von Neumann algebra equipped with a faithful
normal semifinite trace~$\tau$. Let~$A\in\mathcal{M}$ and let
$\{p_n\}_{n\geq0}\subset\mathcal{M}$  be a sequence of
$\tau$-finite projections. If~$\|[A,p_n]\|_2\to0$, then, for every
$m\geq1$, we have
$$\tau((p_nAp_n)^m-A^mp_n)\to0.$$
\end{lemma}
\begin{proof} For~$m=1$ the assertion is obvious. For every~$m\geq 2$, we have
\begin{multline*}
|\tau((p_nAp_n)^m-A^mp_n)|=|\sum_{k=1}^{m-1}p_nA^k(1-p_n)(Ap_n)^{m-k}|\leq\\
\sum_{k=1}^{m-1}\|p_nA^k(1-p_n)\|_2\|(1-p_n)Ap_n\|_2\|A\|_{\infty}^{m-k-1}.
\end{multline*}
Since~$(1-p_n)Ap_n=[A,p_n]p_n$ and~$p_nA^k(1-p_n)=p_n[p_n,A^k]$, and since
$$[p_n,A^k]=\sum_{l=0}^{k-1}A^l[p_n,A]A^{k-1-l},$$
it follows that
\begin{multline*}
|\tau((p_nAp_n)^m-A^mp_n)|\leq \\\sum_{k=1}^{m-1}k\|[p_n,A]\|_2\|A\|_{\infty}^{k-1}\cdot\|[A,p_n]\|_2\cdot
\|A\|_{\infty}^{m-k-1} \\ =\frac{m(m-1)}{2}\|[A,p_n]\|_2^2\|A\|_{\infty}^{m-2}.
\end{multline*}
\end{proof}

\subsection*{The class~$W_1$}

Originally Theorem~\ref{thm:KTFgeneral} is proved for the function of class~$W_1$. That is,
$$W_1=\left\{f\in S'(\mathbb R):\quad\mathcal{F}(f')\in L_1(\mathbb{R})\right\},$$
where~$S'(\mathbb R)$ is the class of all tempered distributions on
$\mathbb{R}$, $\mathcal{F}$ is the Fourier transform and
$L_1(\mathbb{R})$ is the Lebesgue space of all integrable functions on
$\mathbb{R}$. The class~$W_1$ is equipped with the semi-norm
$$\left\|f\right\|_{W_1}=\left\|\mathcal{F}(f')\right\|_1.$$

We need the following simple observation which directly follows
from the fact that the class~$S(\mathbb{R})$ of all Schwartz
functions is dense in~$L_1(\mathbb{R})$.

\begin{lemma}\label{ApproxLemma}
The class of primitives of functions in~$S(\mathbb{R})$ is dense
in~$W_1$, that is for every~$f\in W_1$, there exists a sequence
$f_n\in C^1_b(\mathbb{R})$ such that~$f_n'\in S(\mathbb{R})$ and
$$\lim_{n\rightarrow\infty}\left\|f-f_n\right\|_{W_1}=0.$$
\end{lemma}

\begin{theorem}\label{ps main} Let~$\mathcal{M}$ be a von Neumann algebra equipped with a faithful normal
semifinite trace~$\tau$. If~$A,B$ are unbounded self-adjoint
operators affiliated with~$\mathcal{M}$ such that
$A-B\in\mathcal{L}_1$, then~$f(A)-f(B)\in\mathcal{L}_1$ for every
$f\in W_1$. Moreover, we have
$$\|f(A)-f(B)\|_1\leq\|f\|_{W_1}\|A-B\|_1,$$
$$\tau(f(A)-f(B))=\int_0^1\tau(f'((1-z)A+zB)(A-B))dz$$
for every~$f\in W_1.$
\end{theorem}
\begin{proof} The first inequality is proved by Widom~\cite{Widom}. We present it here for convenience of the reader.
It follows from the equality
$$e^{iC}-e^{iD}=i\int_0^1e^{itC}(C-D)e^{i(1-t)D}dt$$
that
$$\|e^{iC}-e^{iD}\|_1\leq\|C-D\|_1.$$
Hence,
$$\|e^{isA}-e^{isB}\|_1\leq |s|\|A-B\|_1.$$
Recall the obvious equality
$$f(A)-f(B)=\int_{\mathbb{R}}(\mathcal{F}f)(s)(e^{isA}-e^{isB})ds.$$
It follows that
$$\|f(A)-f(B)\|_1\leq\int_{\mathbb{R}}|(\mathcal{F}f)(s)|\cdot\|e^{isA}-e^{isB}\|_1ds\leq$$
$$\leq\int_{\mathbb{R}}|s(\mathcal{F}f)(s)|\cdot\|A-B\|_1ds=\|f\|_{W_1}\cdot\|A-B\|_1.$$

The second equality is proved in~\cite{BS1972-onSSF} (see equation
(2.1) there).
\end{proof}

\section{Krein's theorem in semifinite setting}

The present section proves the following theorem in complete
generality. In the setting ${\mathcal M} ={\mathcal L}({\mathcal
H})$ the result is originally due to M.G. Krein.

\begin{theorem}\label{thm:KTFgeneral}
Let~$\mathcal{M}$ be a~$\sigma$-finite von Neumann algebra with a
faithful normal semifinite trace~$\tau$. If self-adjoint operators
$A,B$ affiliated with~$\mathcal{M}$ are such that
$A-B\in\mathcal{L}_1$, then there is a function~$\xi=\xi_{A, B}\in
L_1(\mathbb{R})$ such that
\begin{equation}\label{KTFgeneral}
\tau(f(A)-f(B))=\int_{\mathbb{R}}f'(s)\xi(s)ds
\end{equation}
for every~$f\in W_1.$  In addition, we have
$$\int_{\mathbb{R}}| \xi(t)|\,dt\leq\left\|V\right\|_1,\quad \int_{\mathbb{R}}\xi(t)dt=\tau(V).$$
\end{theorem}

\subsection*{Proof of Theorem~\ref{thm:KTFgeneral}}

We shall approach the proof of Theorem~\ref{thm:KTFgeneral} via
step by step relaxing of conditions on the trace~$\tau$, function
$f$ and the operators~$A$ and~$B$. This is presented as a series
of lemmas from Lemma~\ref{fin case} to Lemma~\ref{main theorem}.
The final extension to the class~$W_1$ is given in
Lemma~\ref{toW1ext}. For reader's convenience, we denote the
function $\xi_{A, B}$ at different stages with different indices.

We start with rather restrictive case as in the following lemma.
The lemma was noted yet by Krein~\cite{Krein1,Krein2} (see also
p.360 in~\cite{ADS}).

\begin{lemma}\label{fin case}
Suppose that the assumptions of Theorem~\ref{thm:KTFgeneral} hold.
\begin{enumerate}[{\rm (a)}]
\item\label{fina} If~$\tau$ is finite and if~$A,B\in\mathcal{M}$,
then~(\ref{KTFgeneral}) holds with~$\xi=\xi^{(1)}_{A,B}=n_A-n_B$.
The function~$\xi$ is integrable and is supported on
$[0,\max\{\|A\|_{\infty},\|B\|_{\infty}\}]$. \item\label{finb}
Furthermore,
$$\|\xi^{(1)}_{A,B}\|_{\infty}\leq\tau({\rm supp}(A-B)) \ \text{and}\ \ \|\xi^{(1)}_{A,B}\|_1\leq\|A-B\|_1.$$
\item\label{finc} If~$A\geq B$, then~$\|\xi^{(1)}_{A,B}\|_1=\|A-B\|_1$.
\end{enumerate}
\end{lemma}
\begin{proof} (\ref{fina}) Indeed, it follows from the functional calculus that
$$f(A)=-\int_{-\infty}^{\infty}f(s)dE_A(s,\infty).$$
Taking the trace and integrating by parts, we obtain
$$\tau(f(A))=-\int_{-\infty}^{\infty}f(s)dn_A(s)=\int_{-\infty}^{\infty}f'(s)n_A(s)ds.$$
The equation~(\ref{KTFgeneral}) follows immediately. Since~$n_A$
and~$n_B$ are integrable, then so is~$\xi^{(1)}_{A,B}$.

(\ref{finb}) It is clear that~$A\leq B+|A-B|$ and, therefore,
$$n_A(t)\leq n_{A+|B-A|}(t)\stackrel{(\ref{nsum})}{\leq} n_B(t)+n_{|A-B|}(0).$$
Similarly, we have~$B\leq A+|A-B|$ and, therefore,
$$n_B(t)\leq n_{B+|B-A|}(t)\stackrel{(\ref{nsum})}{\leq} n_A(t)+n_{|A-B|}(0).$$
Combining these inequalities, we obtain~$|n_A-n_B|\leq
n_{|A-B|}(0)$, which proves the first assertion. In order to prove
the second inequality, observe that
$$|B-A|=(B-A)_++(B-A)_-\ \ \text{and}\ \ B-A=(B-A)_+-(B-A)_-.$$
Set~$C:=A+(B-A)_+=B-(B-A)_-$. We then have~$C\geq A$ and~$C\geq
B$. In particular, $n_C\geq n_A$ and~$n_C\geq n_B$. Therefore,
$$\|n_C-n_A\|_1=\|n_C\|_1-\|n_A\|_1=\tau(C)-\tau(A)=\tau((B-A)_+),$$
$$\|n_C-n_B\|_1=\|n_C\|_1-\|n_B\|_1=\tau(C)-\tau(B)=\tau((B-A)_-).$$
Hence,
$$\|n_A-n_B\|_1\leq\|n_C-n_A\|_1+\|n_C-n_B\|_1\leq\tau((B-A)_+)+\tau((B-A)_-)=\|A-B\|_1.$$
\end{proof}

The key step in our approach is the extension of Lemma~\ref{fin
case} to the following lemma via approximation process set out in
Theorem~\ref{wvn} and Lemmas~\ref{pr lemma} and~\ref{voic lemma}.

\begin{lemma}\label{main lemma} Suppose that the assumptions of Theorem~\ref{thm:KTFgeneral} hold.
If~$A\geq B\geq 0$, and if~$\tau({\rm supp}(A-B))<\infty$,  then~(\ref{KTFgeneral}) holds for every~$f(s)=s^m$, $m\in\mathbb{N}$,
and for the compactly supported positive function
$\xi=\xi^{(2)}_{A,B}$. If the conditions of Lemma~\ref{fin case}
are also met, then~$\xi^{(1)}_{A,B}=\xi^{(2)}_{A,B}$.
\end{lemma}
\begin{proof} By Theorem~\ref{wvn} and Lemma~\ref{pr lemma},
there exists a family of~$\tau$-finite projections~$p_n$,
$n\geq0$, such that~$p_n\uparrow 1$ and such that
$\|[A,p_n]\|_2\to0$ and~$\|[B,p_n]\|_2\to0$ as~$n\to\infty$. It
follows from Lemma~\ref{voic lemma} that
$$\tau((p_nAp_n)^m-A^mp_n)\to0,\quad\tau((p_nBp_n)^m-B^mp_n)\to0.$$
Hence,
$$\tau((p_nAp_n)^m-(p_nBp_n)^m)\to\tau(A^m-B^m).$$

Set~$a:=\|A\|_{\infty}$. Since~$A\geq B\geq0$ in~$\mathcal{M}$ it
follows that~$p_nAp_n\geq p_nBp_n\geq0$ in the algebra
$p_n\mathcal{M}p_n$. In particular, we have~$n_{p_nAp_n}\geq
n_{p_nBp_n}$ for all~$n\geq0$ in the algebra~$p_n\mathcal{M}p_n$.
By Lemma~\ref{fin case} (\ref{fina}), for every~$n\geq0$, there
exists a positive function~$\xi_n=\xi^{(1)}_{p_nAp_n,p_nBp_n}$
supported on~$[0,a]$ such that
$$\tau((p_nAp_n)^m-(p_nBp_n)^m)=\int_0^ams^{m-1}\xi_n(s)ds.$$
By Lemma~\ref{fin case} (\ref{finb}), we have
$$\|\xi_n\|_{\infty}\leq\tau({\rm supp}(A-B)),\quad \|\xi_n\|_1\leq\|A-B\|_1.$$
Since~$L_{\infty}[0,a]$ is a Banach dual for~$L_1[0,a]$, it
follows from the Banach-Alaoglu theorem that there exists a
directed set~$\mathbb{I}$ and the mapping
$\psi:\mathbb{I}\to\mathbb{Z}_+$ such that
\begin{enumerate}
\item for every~$n\in\mathbb{Z}_+$, there exists~$i(n)\in\mathbb{I}$ such that~$\psi(i)>n$ for~$i>i(n)$.
\item the net~$\xi_{\psi(i)}$, $i\in\mathbb{I}$ converges in weak$^*$ topology.
\end{enumerate}
We set
$$\xi^{(2)}_{A,B}=\lim_{i\in\mathbb{I}}\xi_{\psi(i)}.$$
Therefore,
\begin{multline}\label{tbmd}
\int_0^bms^{m-1}\xi^{(2)}_{A,B}(s)ds=\lim_{i\in I}\int_0^bms^{m-1}\xi_i(s)ds=\\
=\lim_{i\in I}\tau((p_{\psi(i)}Ap_{\psi(i)})^m-(p_{\psi(i)}Bp_{\psi(i)})^m)=\tau(A^m-B^m).
\end{multline}
This shows~(\ref{KTFgeneral}) with~$f(s)=s^m$. The
function~$\xi^{(2)}_{A,B}$ is positive as a weak$^*$-limit of
positive functions.

In order to see that~$\xi^{(1)}_{A,B}=\xi^{(2)}_{A,B}$, we simply write the trace formula~(\ref{KTFgeneral})
$$\int_{\mathbb{R}}ms^{m-1}\xi^{(1)}_{A,B}(s)ds\stackrel{Lemma\ \ref{fin case}}{=}\tau(A^m-B^m)\stackrel{(\ref{tbmd})}{=}\int_{\mathbb{R}}ms^{m-1}\xi^{(2)}_{A,B}(s)ds.$$
Identification now follows from the fact that the polynomials are a separating family of functionals on~$L_1([0,a])$.
\end{proof}

The following lemma removes the positivity assumption on the
operators~$A$ and~$B$ and the assumption that~$\tau({\rm
supp}(A-B))$ is finite.

\begin{lemma}\label{prel} Suppose that the assumptions of Theorem~\ref{thm:KTFgeneral} hold. If~$A,B\in\mathcal{M}$,
 then~(\ref{KTFgeneral}) holds for every~$f \in C^2(\mathbb{R})$ and for the unique compactly supported function
 $\xi=\xi^{(3)}_{A,B}$.  In addition,
$$\|\xi\|_1\leq\|A-B\|_1, \ \ \text{and}\ \ \int_{\mathbb{R}}\xi(s)\,ds=\tau(A-B).$$
If, in addition, the conditions of Lemma~\ref{main lemma} are met,
then~$\xi^{(2)}_{A,B}=\xi^{(3)}_{A,B}$.
\end{lemma}

\begin{proof} Set~$C:=A+(B-A)_+=B-(B-A)_-$. Observe that~$C\geq A$, $C\geq B$ and that~$C-A,C-B\in\mathcal{L}_1$.
Since
$$\tau(f(A)-f(B))=\tau(f(C)-f(B))-\tau(f(C)-f(A)),$$
it is sufficient to consider only the case~$A\geq B\geq 0$.

Clearly, the left hand side of~(\ref{KTFgeneral}) does not depend
on the values of~$f$ outside of the interval
$[-\|A\|_{\infty},\|A\|_{\infty}]$. For every~$f\in
C^2(\mathbb{R})$, there exists~$g\in W_1$ such that~$f=g$ on the
interval~$[-\|A\|_{\infty},\|A\|_{\infty}]$. Thus, we may assume
without loss of generality that~$f\in W_1$.

Let~$0\leq D_n\leq A-B$ be such that~$D_n\uparrow A-B$ and such
that~$\tau({\rm supp}(D_n))<\infty$. It follows from Theorem~\ref{ps main} that
\begin{equation}\label{tbd1}
|\tau(f(B+D_n))-f(A))|\leq{\rm const}\cdot\|f\|_{W_1}\cdot\|B+D_n-A\|_1\to 0.
\end{equation}
Since the functions which are polynomials on
$[-\|A\|_{\infty},\|A\|_{\infty}]$ are dense in~$W_1$, it follows
from Lemma~\ref{main lemma}   that
\begin{equation}\label{tbd2}
\tau(f(B+D_n)-f(B))=\int_0^{\infty}f'(s)\xi^{(2)}_{B+D_n,B}(s)ds
\end{equation}
for every~$f\in W_1$. Hence, we have
$$\int_0^{\infty}f'(s)\xi_{B+D_n,B+D_m}(s)ds\stackrel{Lemma\ \ref{main lemma}}{=}\tau(f(B+D_n)-f(B+D_m))=$$
$$=\tau(f(B+D_n)-f(B))-\tau(f(B+D_m)-f(B))=$$
$$\stackrel{(\ref{tbd2})}{=}\int_0^{\infty}f'(s)(\xi_{B+D_n,B}(s)-\xi_{B+D_m,B}(s))ds.$$
Since~$f'$ is an arbitrary~$C^1$ function on~$[-\|A\|_{\infty},\|A\|_{\infty}]$, it follows that
$$\xi^{(2)}_{B+D_n,B}-\xi^{(2)}_{B+D_m,B}=\xi^{(2)}_{B+D_n,B+D_m}\geq0,\quad n\geq m.$$
Setting~$f(s)=s$ in Lemma~\ref{main lemma}, we obtain that
$\|\xi^{(2)}_{B+D_n,B}\|_1=\|D_n\|_1$. Since
$\|D_n\|_1\leq\|A-B\|_1$, it follows from the Monotone Convergence
Principle, that the sequence~$\xi^{(2)}_{B+D_n,B}$ converges in
$L_1(\mathbb{R})$. Denote its limit by~$\xi^{(3)}_{A,B}$.
Therefore,
$$\tau(f(A)-f(B))\stackrel{(\ref{tbd1})}{=}\lim_{n\to\infty}\tau(f(B+D_n)-f(B))=$$
$$\stackrel{(\ref{tbd2})}{=}\lim_{n\to\infty}\int_0^{\infty}f'(s)\xi^{(2)}_{B+D_n,B}(s)ds
=\int_0^{\infty}f'(s)\xi^{(3)}_{A,B}(s)ds.$$

The identification~$\xi^{(2)}_{A,B}=\xi^{(3)}_{A,B}$ can be established similarly to that in Lemma~\ref{main lemma}.
\end{proof}

The next step in our approach is removing the assumption that the
operators~$A$ and~$B$ are bounded.  We remove this assumption in
two steps: (i)  first we show our construction of a locally
integrable~$\xi$ for unbounded pair of~$A$ and~$B$ (see
Lemma~\ref{FinalStepI}); and (ii) we show that this new~$\xi$ is
 integrable (and positive when $A\ge B$) (see Lemma~\ref{main theorem}).

\begin{lemma}\label{FinalStepI} Suppose that the assumptions of Theorem~\ref{thm:KTFgeneral} hold.
The equality~(\ref{KTFgeneral}) holds for $f\in C_b^1(\mathbb{R})$
such that~$f'\in S(\mathbb{R})$ with the unique locally integrable
function
$$\xi^{(4)}_{A,B}=\xi^{(3)}_{h(A),h(B)}\circ h.$$
Here, $h:\mathbb{R}\to(0,1)$ is a~$C^2$-bijection such that
\begin{equation}\label{hass}
(h^{-1})'(t),(h^{-1})''(t)=O((h^{-1}(t))^n),\quad t\to 0,1,
\end{equation}
for some~$n<\infty$. If~$A,B$ are bounded, then~$\xi^{(4)}_{A,B}=\xi^{(3)}_{A,B}$.
\end{lemma}
\begin{proof} As in Lemma~\ref{prel}, we may assume that~$A\geq B$.
Applying Theorem~\ref{ps main} to the operators~$A$, $B$ and to the function~$h$, we obtain
\begin{equation*}
\|h(A)-h(B)\|_1 \leq{\rm const} \cdot \|A-B\|_1.
\end{equation*}
Define the function~$g$ by setting~$g:=f\circ h^{-1}$. It is a
composition of two~$C^2$ functions and, therefore, $g\in
C^2(-1,1)$. The assumption~$f'\in S(\mathbb{R})$ guarantees that
$f',f''$ decrease faster than any power function. This together
with~(\ref{hass}) implies that
$$g'(+0)=g'(1-0)=0,\quad g''(+0)=g''(1-0)=0.$$
Hence, we have~$g\in C^2[0,1]$. Clearly, $g$ extends to a function
$g\in C^2_b(\mathbb{R})$. Applying Lemma~\ref{prel} to the
operators~$h(A)$ and~$h(B)$, we now obtain
\begin{multline}\label{FinalStepIeq}
\tau(f(A)-f(B))=\tau(g(h(A))-g(h(B)))=\\
=\int_{\mathbb{R}}g'(s)\xi^{(3)}_{h(A),h(B)}(s)ds=\int_{\mathbb{R}} f'(u) \xi^{(3)}_{h(A),h(B)}(h(u))du,
\end{multline}
where in the last step we substituted~$s=h(u)$. This
proves~(\ref{KTFgeneral}). Since the left hand side does not
depend on~$h$ and since~$f'$ is an arbitrary Schwartz function,
the uniqueness follows. The local integrability of
~$\xi^{(4)}_{A,B}$ follows from the integrability of
~$\xi^{(3)}_{h(A),h(B)}$ combined with the assumption that ~$h$ is
a~$C^2$-bijection.


If~$A$ and~$B$ are bounded, then by using the trace formula for the left hand side of~(\ref{FinalStepIeq}), we obtain
$$\int_{\mathbb{R}}f'(u)\xi^{(3)}_{A,B}(u)du=\tau\left(f(A)-f(B)\right)=\int_{\mathbb{R}}f'(u)\xi^{(4)}_{A,B}(u)du.$$
The latter implies that~$\xi^{(3)}_{A,B}=\xi^{(4)}_{A,B}$ in this case.
\end{proof}

In what follows, $C_b(\mathbb{R})$ is the space of all bounded
functions in~$C(\mathbb{R})$. The spaces~$C_b^1(\mathbb{R})$ and
$C_b^2(\mathbb{R})$ are defined in an obvious manner.

\begin{lemma}\label{main theorem} Suppose that the assumptions of Theorem~\ref{thm:KTFgeneral} hold.
The equality~(\ref{KTFgeneral}) holds with~$\xi^{(4)}_{A,B}$ of Lemma~\ref{FinalStepI},
 for every~$f\in C^1_b(\mathbb{R})$ such that~$f'\in S(\mathbb{R})$. Moreover, $\xi^{(4)}_{A,B}\in L_1(\mathbb{R})$.
 If~$A\geq B$, then we also have~$\xi^{(4)}_{A,B}\geq0$.
\end{lemma}
\begin{proof} Without loss of generality, $A\geq B$. We show that the function of Lemma~\ref{FinalStepI}
is positive and integrable.

We show the positivity first. Let~$\alpha>0$ and let
$a<b\in\mathbb{R}$. Define the function~$h_{a,b,\alpha}$ by
setting~$h_{a,b,\alpha}(-\infty)=0$ and
$$h_{a,b,\alpha}'(t)=\begin{cases}
1,\quad t\in[a,b]\\
\frac{\alpha^3}{(\alpha^2+(t-b)^2)^{3/2}},\quad t>b\\
\frac{\alpha^3}{(\alpha^2+(t-a)^2)^{3/2}},\quad t<a
\end{cases}
$$
The function~$h_{a,b,\alpha}$ is a~$C^2$-bijection
$\mathbb{R}\mapsto(0,c)$ where~$c=h_{a,b,\alpha}(+\infty)$. A
simple computation shows that~$h_{a,b,\alpha}$ satisfies the
condition~(\ref{hass}) in Lemma~\ref{FinalStepI}. If
$\xi^{(4)}_{A,B}$ is from Lemma~\ref{FinalStepI}, then
$$\xi^{(4)}_{A,B}=\xi^{(3)}_{h_{a,b,\alpha}(A),h_{a,b,\alpha}(B)}\circ h_{a,b,\alpha}.$$

Combining Theorem~\ref{ps main} with the well known fact that
$\tau(CD)\ge 0$ when $0\leq C\in \mathcal {M}$ and $0\leq D\in
\mathcal{L}_1$, we obtain the following estimate
$$\tau(h_{a,b,\alpha}(A)-h_{a,b,\alpha}(B))=\int_0^1\tau(h_{a,b,\alpha}'((1-z)A+zB)(A-B))dz\geq0.$$
Furthermore, by Lemma~\ref{prel} (applied to~$h_{a,b,\alpha}(A)$
and~$h_{a,b,\alpha}(B)$ and~$f(s)=s$) and substituting
$s=h_{a,b,\alpha}(u)$, we obtain
\begin{multline*}
\tau(h_{a,b,\alpha}(A)-h_{a,b,\alpha}(B))=\int_0^c\xi^{(3)}_{h_{a,b,\alpha}(A),h_{a,b,\alpha}(B)}(s)ds=\\
=\int_{\mathbb{R}}\xi^{(3)}_{h_{a,b,\alpha}(A),h_{a,b,\alpha}(B)}(h_{a,b,\alpha}(u))h_{a,b,\alpha}'(u)du
=\int_{\mathbb{R}}\xi^{(4)}_{A,B}(u) h_{a,b,\alpha}'(u)du.
\end{multline*}
In particular, the function~$\xi^{(4)}_{A,B}h'_{a,b,\alpha}$ is
integrable.  Since~$h_{a,b,\alpha}'\to\chi_{(a,b)}$ almost
everywhere as~$\alpha\to0$ and since~$h_{a,b,\alpha}'\leq
h_{a,b,1}'$ when~$\alpha<1$, it follows from the Dominated
Convergence Principle that
$$\int_a^b\xi^{(4)}_{A,B}(u)du = \lim_{\alpha\to0}\int_{\mathbb{R}}\xi^{(4)}_{A,B}(u)h_{a,b,\alpha}'(u) du\geq 0.$$
Since the latter inequality holds for arbitrary scalars~$a,b$, it follows that~$\xi^{(4)}_{A,B}\geq0$.

Now we show the integrability of~$\xi^{(4)}_{A,B}$ on
$\mathbb{R}$. Consider the~$C^2$-bijection
$h_{\alpha}:\mathbb{R}\to(-1,1)$ given by
$$h_\alpha(s)=\frac{s}{(\alpha^2+s^2)^{\frac12}},\quad s\in\mathbb{R}.$$
If~$\xi^{(4)}_{A,B}$ is from Lemma~\ref{FinalStepI}, then
$$\xi^{(4)}_{A,B}=\xi^{(3)}_{h_{\alpha}(A),h_{\alpha}(B)}\circ h_{\alpha}.$$
By Lemma~\ref{prel} (applied to the operators~$h_{\alpha}(A)$ and~$h_{\alpha}(B)$) we have
\begin{multline*}
\tau(h_{\alpha}(A)-h_{\alpha}(B))=\int_{-1}^1\xi^{(3)}_{h_{\alpha}(A),h_{\alpha}(B)}(s)ds=\\
=\int_{\mathbb{R}}\xi^{(3)}_{h_{\alpha}(A),h_{\alpha}(B)}(h_{\alpha}(u))h_{\alpha}'(u)du
=\int_{\mathbb{R}}\xi^{(4)}_{A,B}(u)h_{\alpha}'(u)du.
\end{multline*}
We have
\begin{multline*}
\left|\int_{-\infty}^{\infty}\xi^{(4)}_{A,B}(u)\cdot h_{\alpha}'(u)du \right|=
\left|\tau(h_{\alpha}(A)-h_{\alpha}(B))\right|\leq\\
\leq\left\| h_\alpha(A)-h_\alpha(B)\right\|_1\leq\frac{\mathrm{const}}{\alpha}\left\|A-B\right\|_1,
\end{multline*}
where the last estimate follows from Theorem~\ref{ps main}, applied to the operators~$\alpha^{-1}A$ and
$\alpha^{-1}B$ and function~$h_1$ (note that~$h_{\alpha}(s)=h_1(\frac{s}{\alpha})$ for all~$s$). Therefore, we have
$$\alpha\int_{\mathbb{R}}h_{\alpha}'(u)\xi_{A,B}^{(4)}(u)du=O(1).$$
Since~$\alpha h_{\alpha}'\uparrow 1$ when~$\alpha\to\infty$, we infer from the
Monotone Convergence Principle (which is applicable since~$\xi^{(4)}_{A,B}\geq0$) that~$\xi^{(4)}_{A,B}$ is integrable.
\end{proof}

Finally we give a simple extension of Lemma~\ref{main theorem} to
the class~$W_1$. The extension is based on our earlier observation
in Lemma~\ref{ApproxLemma}.

\begin{lemma}\label{toW1ext} Suppose that the assumptions of Theorem~\ref{thm:KTFgeneral} hold.
The trace formula~(\ref{KTFgeneral}) holds for every~$f\in W_1$.
\end{lemma}
\begin{proof} Fix~$f\in W_1$. By Lemma~\ref{ApproxLemma},
there is a sequence~$f_n\in C_b^1(\mathbb R)$ such that~$f_n'\in S(\mathbb R)$ and such that
$$\lim_{n\rightarrow\infty}\left\|f-f_n\right\|_{W_1}=0.$$

By Lemma~\ref{main theorem}, we have
\begin{equation}\label{ext0}
\tau(f_n(A)-f_n(B))=\int_{\mathbb{R}}f'(s)\xi(s)ds.
\end{equation}

By Theorem~\ref{ps main}, we have
$$|\tau(f_n(A)-f_n(B))-\tau(f(A)-f(B))|\leq\|(f_n-f)(A)-(f_n-f)(B)\|_1\leq$$
$$\leq{\rm const}\cdot\|f_n-f\|_{W_1}\|A-B\|_1.$$
Therefore,
\begin{equation}\label{ext1}
\lim_{n\to\infty}\tau(f_n(A)-f_n(B))=\tau(f(A)-f(B)).
\end{equation}

On the other hand, we have~$\xi\in L_1$ and, therefore,
$\mathcal{F}(\xi)\in L_{\infty}$. We have
$\mathcal{F}(f_n')\to\mathcal{F}(f)$ in~$L_1$ and, therefore,
$$\mathcal{F}(f_n')\mathcal{F}(\xi)\to\mathcal{F}(f)\mathcal{F}(\xi)$$
in~$L_1$. It follows that
\begin{multline}\label{ext2}
\int_{\mathbb{R}}f_n'(s)\xi(s)ds=\int_{\mathbb{R}}(\mathcal{F}(f_n'))(s)(\mathcal{F}(\xi))(s)ds\to\\
\to\int_{\mathbb{R}}(\mathcal{F}(f'))(s)(\mathcal{F}(\xi))(s)ds=\int_{\mathbb{R}}f'(s)\xi(s)ds.
\end{multline}

Combining~(\ref{ext0}), (\ref{ext1}) and~(\ref{ext2}), we conclude the proof.
\end{proof}

\begin{remark}
  We observe that Theorem~\ref{thm:KTFgeneral} and the trace formula
  (\ref{KTFgeneral}) can be further extended from the class of
  functions~$W_1$ to the homogeneous Besov class~$\tilde
  B^1_{\infty1}$ (see e.g.~\cite{Peller1990-MR1044807}).  In the case
  when~$\mathcal{M}={\mathcal L}(H)$, such an extension is performed in
  \cite{Peller1990-MR1044807}, in the general case the argument is
  exactly the same as in~\cite{Peller1990-MR1044807}. We omit further
  details.
\end{remark}


\begin{thebibliography}{100}

\bibitem{ADS} Azamov N., Dodds P., Sukochev F. {\it The Krein spectral shift function in semifinite von Neumann algebras.} Integral Equations Operator Theory {\bf 55} (2006), no. 3, 347--362.
\bibitem{BirPus1998-MR1607900} Birman M., Pushnitski A. {\it Spectral shift function, amazing and multifaceted.} Integral Equations Operator Theory {\bf 30} (1998), no. 2, 191--199, Dedicated to the memory of Mark Grigorievich Krein (1907--1989).
\bibitem{BS1972-onSSF} Birman M., Solomyak M. {\it Remarks on the spectral shift function.} Zap. Nauchn. Sem. Leningrad. Otdel. Mat. Inst. Steklov. (LOMI) {\bf 27} (1972), 33--46, Boundary value problems of mathematical physics and related questions in the theory of functions, 6.
\bibitem{BS-DOI-I} Birman M., Solomyak M. {\it Double Stieltjes operator integrals.} Problemy Mat. Fiz. (1966), no. 1, 33--67, {\it Russian}.
\bibitem{BS-DOI-II} Birman M., Solomyak M. {\it Double Stieltjes operator integrals. II} Problemy Mat. Fiz. (1967), no. 2, 26--60, {\it Russian}.
\bibitem{BS-DOI-III} Birman M., Solomyak M. {\it Double Stieltjes operator integrals. III} Problemy Mat. Fiz. (1973), no. 6, 27--53, {\it Russian}.
\bibitem{BratRob} Bratteli O., Robinson D. {\it Operator algebras and quantum statistical mechanics. 1. $C^*$- and~$W^*$-algebras, symmetry groups, decomposition of states.} Second edition. Texts and Monographs in Physics. Springer-Verlag, New York, 1987.
\bibitem{Davidson} Davidson K. {\it $C^*$-algebras by example.} Fields Institute Monographs, 6. American Mathematical Society, Providence, RI, 1996.
\bibitem{Krein1} Krein M. {\it On the trace formula in perturbation theory.} (Russian) Mat. Sbornik N.S. {\bf 33} (75), (1953). 597--626.
\bibitem{Krein2} Krein M. {\it On some new investigations in perturbation theory.} First Math. Summer
School, Kiev (1963), 104--183.
\bibitem{dPS}  de Pagter B., Sukochev F. {\it Differentiation of operator functions in non-commutative~$L_p$-spaces.} J. Funct. Anal. {\bf 212} (2004), no. 1, 28--75.
\bibitem{Peller2005-MR2124872} Peller V. {\it An extension of the Koplienko-Neidhardt trace formulae.} J. Funct. Anal. {\bf 221} (2005), no. 2, 456--481.
\bibitem{Peller1990-MR1044807} Peller V. {\it Hankel operators in the perturbation theory of unbounded selfadjoint operators.} Analysis and partial differential equations, Lecture Notes in Pure and Appl. Math., vol. 122, Dekker, New York, 1990, pp. 529--544.
\bibitem{PX}  Pisier G., Xu Q. {\it Non-commutative~$L_p$-spaces.} Handbook of the geometry of Banach spaces, Vol. 2, 1459–1517, North-Holland, Amsterdam, 2003.
\bibitem{PS}  Potapov D., Sukochev F. {\it Unbounded Fredholm modules and double operator integrals.} J. Reine Angew. Math. {\bf 626} (2009), 159--185.
\bibitem{PSW} de Pagter B., Witvliet H., Sukochev F. {\it Double operator integrals.} J. Funct. Anal. {\bf 192} (2002), no. 1, 52--111.
\bibitem{Simon} Simon B. {\it Spectral averaging and the Krein spectral shift.} Proc. Amer. Math. Soc. {\bf 126} (1998), no. 5, 1409--1413.
\bibitem{SZ-Lec-on-vNA}{\c{S}}.~Str{\u{a}}til{\u{a}} and L.~Zsid{\'o}, \emph{Lectures on
von {N}eumann algebras}, Editura Academiei, Bucharest, 1979.
\bibitem{SZ} Sukochev F., Zanin D. {\it Johnson-Schechtman inequalities in the free probability
theory.} J. Funct. Anal. {\bf 263} (2012), no. 10, 2921-2948.
\bibitem{Voiculescu} Voiculescu D. {\it On a trace formula of M. G. Krein.} Operators in indefinite metric spaces, scattering theory and other topics (Bucharest, 1985), 329--332, Oper. Theory Adv. Appl., 24, Birkhäuser, Basel, 1987.
\bibitem{Widom} Widom H. {\it When are differentiable functions differentiable?} in Linear and complex analysis problem book, Lect. Notes Math., {\bf 1043} (1984) 184--188.
\end{thebibliography}
\end{document}